\documentclass[a4paper]{amsart}

\usepackage{amsfonts, latexsym, amssymb, amsthm, amsmath, mathrsfs, esint}

\newcommand{\R}{\mathbb{R}}
\newcommand{\N}{\mathbb{N}}
\newcommand{\Z}{\mathbb{Z}}
\newcommand{\C}{\mathbb{C}}

\newcommand{\St}{\mathbb{S}^2}

\newcommand{\scp}[2]{\left\langle #1, #2 \right\rangle}
\newcommand{\dd}[2]{\frac{\partial #1}{\partial #2}}

\newcommand{\m}{\boldsymbol{m}}

\newcommand{\eps}{\varepsilon}
\newcommand{\del}{\partial}

\newcommand{\norm}[2]{\left\| #1 \right\|_{ #2 }}
\newcommand{\upper}[2]{ #1^{(#2)}}

\renewcommand{\Re}{{\rm Re}}
\renewcommand{\Im}{{\rm Im}}

\renewcommand{\div}{\mathop{\mathrm{div}}\nolimits}

\newtheorem*{theorem}{Theorem}
\newtheorem{lemma}{Lemma}
\newtheorem{proposition}{Proposition}
\newtheorem*{definition}{Definition}

\newtheorem*{notation}{Notation}
\newtheorem*{remark}{Remark}

\author[C. Melcher]{Christof Melcher}
\address{Department of Mathematics I \\ RWTH aachen University \\52056 aachen \\ Germany}
\email{melcher@rwth-aachen.de}
\title[Global solvability of Landau-Lifshitz-Gilbert]{Global solvability of the Cauchy problem for the Landau-Lifshitz-Gilbert equation in
higher dimensions}
\date{\today}

\begin{document}
\maketitle

\begin{abstract}
We prove existence, uniqueness and asymptotics of global smooth solutions for the Landau-Lifshitz-Gilbert
equation in dimension $n \ge 3$, valid under a smallness condition of initial gradients in the $L^n$ norm. 
The argument is based on the method of moving frames that produces a covariant complex Ginzburg-Landau equation,
and a priori estimates that we obtain by the method of weighted-in-time norms as introduced by Fujita and Kato.
\end{abstract} 

\section*{Introduction and results}
The Landau-Lifshitz-Gilbert equation (LLG) is the fundamental evolution law in magnetism. It governs
the dynamics of continuous spin systems, i.e. of director fields $\m= (0,\infty) \times \R^n \to \mathbb{S}^2$
with values in the unit sphere $\St \subset \R^3$ where typically $n=3$.
In its original form as introduces by Landau and Lifshitz \cite{LandauLifshitz:35, Gilbert:04} the equation reads
\begin{equation}\label{eq:LLG}
\dd{\m}{t} = - \m \times \Delta \m -  \lambda \, \m \times  \m \times \Delta \m 
\end{equation}
where $\lambda>0$ is a damping parameter and $\times$ is the vector product on $\R^3$. Observe that 
$-\m \times  \m \times \Delta \m = \Delta \m +|\nabla \m|^2 \m$, which is sometimes called the tension field.
 Mathematically LLG can be considered as a dissipative version
of the Schr\"odinger flow for harmonic maps into $\St$ (where $\lambda=0$). Emphasizing its parabolic character, LLG
can also be considered as a quasilinear perturbation of the heat flow for harmonic maps
by the (conservative) precession term $-\m \times \Delta \m$. 
A common Liapunov functional is given by the Dirichlet energy
$$E(\m)=\frac{1}{2} \int_{\R^n} |\nabla \m|^2 \, dx.$$
In dimension $n=2$, LLG is energy critical, which means that the scaling symmetry $\m_\eps(t,x)=\m(t/ \eps^2  ,x/ \eps)$
for $\eps >0$ featured by \eqref{eq:LLG} preserves the Dirichlet energy, i.e. $E(\m_\eps(t))=E(\m(t/\eps^2))$ for all $t>0$.
Global existence of weak solutions and uniqueness in the class of energy decreasing solutions
can be shown with arguments parallel to those for the heat flow, see \cite{Guo-Hong:93, Harpes:04, Harpes:06}. 
In contrast, the existence of finite time singularities as known for the heat flow has not been shown conclusively, but there is at least strong numerical evidence, see \cite{Bartels}. 
The possible blow-up cenario, however, is precisely characterized through the development of bubbles at energy concentration points. 
Consequently, initial energies below $4 \pi$, which is the topological lower bound for a full cover of $\St$,
will guarantee global regularity. 
Large energy solutions on the other hand are particularly interesting if topologically nontrivial data, featuring e.g. magnetic vortices, is concerned. 
Recent results show that in the presence of strong potentials of Ginzburg-Landau type, conventional energy bounds can be bypassed 
by certain well-preparedness assumptions on the initial data that prevent the formation of (extra) bubbles, see \cite{Kurzke-Melcher-Moser-Spirn:11, Kurzke-Melcher-Moser:11}.\\
In dimension $n \ge 3$, LLG becomes super-critical with respect to the Dirichlet energy, i.e. with $\m_\eps$ as before $E(\m_\eps(t)) \to 0$ 
as $ \eps \to 0$.
Whereas the proof of existence of weak solutions remains unaffected, see \cite{Alouges_Soyeur:92}, questions of uniqueness and regularity become more delicate. In particular considerations solely based on the energy are insufficient for ruling out concentration effects, and new adapted quantities 
have to come into play. More advanced features of the harmonic map heat flow equation, such as 
Bochner identities and monotonicity formulas giving rise to crucial a priori estimates, see \cite{Struwe:88}, are not at hand for LLG.
In dimension $n=3,4$ perturbative methods based on elliptic problems on suitable time slices provide local regularity conditions for LLG
in terms of certain scaling invariant space-time norms of Morrey-type, see \cite{Moser:book, Melcher:05, Wang:06}. This auxiliary
approach leads to partial regularity results parallel to those for the harmonic map heat flow in \cite{Struwe:88, Chen_Struwe:89}, 
but it fails to apply in higher dimensions $n>4$. Hence the investigation of more flexible and conceptually more adapted approaches 
would be desirable. In this note we establish a relationship between LLG and covariant complex Ginzburg-Landau equations. 
This observation is inspired by recent developments in the context of Schr\"odinger maps. 
It enables us to prove global existence, uniqueness and regularity in arbitrary 
dimensions $n \ge 3$ under a smallness condition of initial data with respect to the scaling 
invariant homogoneneous $\dot{W}^{1,n}$-Sobolev norm. Here is our main results:

\begin{theorem} \label{thm:2}Suppose $\lambda>0$ and $n \ge 3$. Then there exist constants $\rho>0$ and $c>0$
with the following property:
Given  $\m_\infty \in \St$ and initial data $ \m_0 : \R^n \to \St$ such that 
$$\m_0-\m_\infty \in H^1 \cap W^{1,n} (\R^n;\R^3)$$
and such that 
$$ \| \nabla \m_0 \|_{L^n} < \rho,$$
then there exists a global smooth solution $$\m:(0,\infty) \times \R^n \to \St$$
for the Landau-Lifshitz-Gilbert equation \eqref{eq:LLG} with the properties that
$$ \sup_{t>0} \|  \m(t) - \m_\infty \|_{H^1} \le  \|  \m_0 - \m_\infty\|_{H^1}$$
and
$$  \sup_{t>0} \sqrt{t}\  \| \nabla \m(t) \|_{L^\infty} +  \sup_{t>0} \| \nabla \m(t) \|_{L^n} \le c \; \| \nabla \m_0\|_{L^n}$$
and such that 
$$\lim_{t \searrow 0}\left( \m(t) -\m_0 \right)=0 \quad \text{strongly in} \quad H^1 \cap W^{1,n}(\R^n;\R^3).$$

The solution is unique in its class and satifies $\displaystyle{\lim_{t \to \infty} \m(t) = \m_\infty}$ in $C^1(\R^n;\R^3)$.\\

If in addition, $ \m_0 - \m _\infty \in H^{\sigma}(\R^n;\R^3)$ 
for some integer  $\sigma > \frac{n}{2}+1$,
then $$\m-\m_\infty \in C^0\left( [0,\infty);H^{\sigma}(\R^n;\R^3) \right) \cap C^0\left( (0,\infty);H^{\infty}(\R^n;\R^3) \right).$$ 
\end{theorem}

%

Concerning local solutions, existence and uniqueness for $H^\sigma$-regular initial data (with $\sigma$ as in the Theorem)
has been shown in \cite{Ding_Wang:01, Kenig:10} in more generality. Observe that in particular 
$H^\sigma(\R^n) \hookrightarrow {\rm Lip}(\R^n)$, and local Lipschitz bounds play indeed a crucial role in constructing classical solutions for
strongly parabolic systems of that type.
In the context of heat flows local Lipschitz bounds are provided by the Bochner identity in conjunction with the maximum principle, and it is
possible to prove local existence and uniqueness results under weaker assumptions.\\
The main point in extending smooth solutions or proving global regularity of weak solutions is to show that the smallness condition in 
$\dot{W}^{1,n}(\R^n)$ implies a Lipschitz bound. In the context of heat flows into an arbitrary compact Riemannian manifold, a global regularity result similar to our 
Theorem has been proven in \cite{Struwe:88} Theorem 7.1
assuming a global Lipschitz bound and a corresponding smallness in energy condition on initial data. \\
Recall that smoothness and a separate smallness in energy condition on initial data alone is not enough to ensure global regularity,
as has been shown for LLG at least in dimension $n=3,4$ in \cite{Ding_Wang:07}. 
In the context of the Theorem it would be interesting to investigate whether a bound on the gradient $\nabla \m$ in
$L^\infty((0,T); L^n(\R^n))$ implies regularity of an arbitrary weak solution $\m$. In case of heat flows this is a
recent result, see \cite{Wang:08}, and motivated by parallel results for
the Navier-Stokes equation, see \cite{Sverak:03}. The argument, however, crucially rests on specific features of
the heat flow equation.

%

%
%
%

In the case of Schr\"odinger maps, i.e. for $\lambda=0$, the parabolicity along with its characterisitc local smoothing properties
degenerates. Recent global existence and regularity results that rest on extensive arguments from Fourier analysis hold true 
under a smallness condition in terms of critical Sobolev norm $\dot{H}^\frac{n}{2}(\R^n)$, see  \cite{Bejenaru:07, Bejenaru:08}, rather than the 
weaker $\dot{W}^{1,n}(\R^n)$ norm. Thus, in contrast to the present result for LLG, results for Schr\"odinger maps in higher dimensions
take into account higher order (fractional) derivatives, 
similar to the case of wave maps, see \cite{Shatah-Struwe:02, Tao:04}.

The crucial point and common feature in the proof for LLG and for Schr\"odinger maps is a canonical choice of coordinates on the tangent
bundle (rather than on the target manifold itself)
called the method of moving frames. This method has several advantages. First, it reveals the essentially cubic 
structure of the geometric nonlinearity. More importantly in the context of LLG, however,
it provides a linearization of the quasilinear hybrid-structure and turns the equation
into a semilinear (though nonlocal) complex Ginzburg-Landau equation. Finally, the method is global, and
no small deviation assumption as for the stereographic model is required.
%
The strategy of proof is to start from smooth initial data and local classical solutions that can be constructed by means of general methods for
quasilinear parabolic systems. Crucial higher order energy estimates will be reviewed in Section 1, where we also address the issue of uniqueness.
In Section 2 we shall briefly discuss the concept of moving frames. We derive a covariant version of LLG giving rise to a covariant complex Ginzburg-Landau equation in terms
of a complex vector field $u$ that represents the spatial gradient of $\m$ and a space-time connection field ${\bf a}=(a_0,a)$ that 
represents a covariant  derivative. Using the Coulomb gauge $\div a =0$ we shall deduce some preliminary 
estimates for the nonlinearities coming out of this process. In Section 3 we shall derive corresponding a priori estimate for $u(t)$
in the $L^n(\R^n)$ and $L^\infty(\R^n)$ norm
that hold true under a smallness condition on $\|u(0)\|_{L^n}$. Whereas a priori estimates for Sch\"odinger maps mainly rely on 
sophisticated Paley-Littlewood decompositions and Strichartz estimates for dispersive equation, 
see  \cite{Bejenaru:07, Bejenaru:08},  
a priori estimates for LLG will rest on methods from semilinear parabolic systems
and in particular the usage of weighted-in-time Lebesgue-Sobolev spaces as introduced by Fujita and Kato in the context of the 
Navier-Stokes equation, see \cite{Fujita_Kato:64, Kato:84}.
Accordingly, we obtain in the first instance uniform estimates for quantities $t \mapsto t^\frac{1-\delta}{2} \| u(t) \|_{L^\frac{n}{\delta}}$
for certain $\delta < 1$
and $t \mapsto \sqrt{t} \| \nabla u(t) \|_{L^{n}}$ featuring the appropriate scaling in space-time but improved regularity in space.
By Duhamel's principle, these quantities can be 
controlled by $\|u(0)\|_{L^n}$, and we use them to obtain 
$$ \sqrt{t} \; \|u(t)\|_{L^\infty} + \| u(t) \|_{L^n} \le c \, \|u(0)\|_{L^n}$$
with a universal constant $c$ and independently of $t$.
Transferring this estimate back into LLG  in Section 4
provides in particular a Lipschitz bound for $\m$, which enables us to extend the local classical solutions $\m$ uniquely to all times and, by approximation, to prove regularity of weak solutions.   
 
\section{Local existence and uniqueness of classical solutions}
Local smooth solutions for LLG have been used in \cite{Ding_Wang:01} with the aim of 
constructing Schr\"odinger maps in the limit  $\lambda \searrow 0$. More recently, an alternative approach based
on fourth order parabolic approximation has been presented in \cite{Kenig:10}. The approaches work in the context
of more general target manifolds equipped with a K\"ahler structure. Here we shall restrict to the case of $\St$ valued maps 
where local classical solutions can be constructed extrinsically by means of standard methods for quasilinear parabolic systems. 
Our main point is the fact that classical solutions persist as long as spatial gradients  remain bounded.\\

We shall fix constant values for $\m$ near infinity, i.e. 
$\m(\infty)= \m_\infty$ for some fixed $\m_\infty \in \St$.  For this purpose we define, for $\sigma \in \N$, the complete metric spaces
$$ H^\sigma_\ast(\R^n)=H^\sigma_\ast(\R^n; \St)=\{ \m : \R^n \to \St : \m- \m_\infty   \in H^{\sigma}(\R^n;\R^3)\},$$
where $H^\sigma(\R^n)=W^{\sigma,2}(\R^n)$ the usual Sobolev space. The distance between two elemensts $\m_1$ and $\m_2$ in  
$H^\sigma_\ast(\R^n)$ is given by $\|\m_1 -\m_2\|_{H^{\sigma }}$. We also define the space 
$$H^\infty_\ast(\R^n; \St)= \bigcap_{\sigma \in \N} H^\sigma_\ast(\R^n; \St)$$
which is modeled on the Fr\'echet space $H^\infty(\R^n)$ equipped with the usual metric. \\


For initial data $\m_0 \in H^\sigma _\ast$ and $\sigma$ sufficiently large we shall briefly discuss short time and global existence and 
uniqueness of classical solutions.  In view of the vector identities
$$ \m \times \Delta \m =  \nabla \cdot ( \m \times \nabla \m)$$
and 
$$- \m \times \m \times \Delta \m =  \Delta \m +|\nabla \m|^2 \m$$
 valid for $\m \in C^2(\R^n;\St)$,  equation \eqref{eq:LLG} can equivalently be written as 
\begin{equation}\label{eq:LLG2}
\dd{\m}{t} = \lambda \left( \Delta \m + |\nabla \m|^2 \m \right)  - \m \times \Delta \m
\end{equation}
which is a quasilinear parabolic system. Taking into account 
$$\m \times \Delta \m =  \nabla \cdot (\m \times \nabla \m)  $$
we observe that its principle part is in divergence form and induces a concept of weak solutions:
\begin{definition}
We say  
$\m \in L^\infty((0,T); H^1_\ast (\R^n;\St))$ with  $\dd{\m}{t} \in L^2((0,T); L^2(\R^n;\R^3))$
is a weak solutions of Landau-Lifshitz-Gilbert in $(0,T)$ if  
$$\left\langle \dd{\m}{t}, \Phi \right\rangle_{L^2} + \lambda   \langle \nabla \m , \nabla \Phi \rangle_{L^2}  
=  \langle \m \times \nabla \m , \nabla \Phi \rangle_{L^2}  + \lambda \left\langle |\nabla \m|^2 \m, \Phi \right\rangle_{L^2}  $$
for every $ \Phi \in H^1\cap L^\infty (\R^n;\R^3)$ and almost every $t \in (0,T)$.
\end{definition}
It is useful to take into account several variants of LLG. Vector multiplication of \eqref{eq:LLG2} by $\m$ from the left
and adding a multiple of \eqref{eq:LLG2} yields
\begin{equation}\label{eq:LLG3}
 \lambda \, \dd{\m}{t}  +  \m \times \dd{\m}{t} = (1+ \lambda^2) \left( \Delta \m + |\nabla \m|^2 \m \right)
\end{equation}
from which one can easily deduce the energy (conservation) law 
\begin{equation}\label{eq:energy_conservation}
(1+\lambda^2) \, \frac{d}{dt} E(\m(t)) + \lambda \int_{\R^n} \left|\dd{\m(t)}{t} \right|^2 \, dx =0
\end{equation}
valid for sufficiently regular solutions, hence 
\begin{equation}\label{eq:energy_inequality}
 E(\m(t)) - E(\m(0))=  - \frac{\lambda}{1+\lambda^2} \int_0^t \int_{\R^n} \left|\dd{\m}{t} \right|^2 \, dx \, ds.  
\end{equation}
The corresponding a priori estimates are sufficient for constructing global weak solutions, for which \eqref{eq:energy_inequality}
turns into an inequality, see \cite{Alouges_Soyeur:92}. Higher order energy estimates, however, are necessary for proving short time existence of classical solutions:

\begin{proposition} \label{prop:local_classic}
Suppose 
$\sigma > \frac{n}{2}+1$ is an integer and $\m_0 \in H^\sigma_\ast(\R^n;\St)$. Then there exists a time
$T>0$ and a classical solution $\m:(0,T) \times \R^n \to \St$
of  the Landau-Lifshitz-Gilbert equation \eqref{eq:LLG} with initial data $\m_0$ such that 
$$ \m \in C^0\left([0,T); H^\sigma_\ast (\R^n)\right) \quad \text{and} \quad 
 \dd{\m}{t}  \in C^0\left([0,T); H^{\sigma-2}(\R^n)\right).$$
The solution persists as long as  $\|\nabla \m(t)\|_{L^\infty} $ remains finite.
\end{proposition}

\begin{remark} The terminal time $T$ can be bounded below in terms of $\| \m_0 - \m_\infty \|_{H^{\sigma_0}}$,
where $\sigma_0$ is the smallest integer greater than $\frac{n}{2}+1$.\\
The result holds true for $\sigma = \infty$. For less regular initial maps 
$\m_0 \in H^\sigma_\ast$ with
$\sigma > \frac{n}{2}+1$ we still have solutions
$$ \m \in C^0\left((0,T); H^\infty_\ast (\R^n)\right)$$
in particluar, solution from Proposition \ref{prop:local_classic} are in $C^\infty((0,T) \times \R^n)$.
\end{remark}

A general strategy for solving quasilinear systems by means of spectral truncation and higher order energy estimates 
has been presented in \cite{Taylor_pseudo} \S 7 or \cite{Taylor_3} Chapter 15 \S 7, respectively.  Since \eqref{eq:LLG}
respects the constraint $|\m|=1$, it will be sufficient to construct a solution in a usual Hilbert space by virtue of the ansatz  
$\m ={\bf \hat{e}}_3 + v$. Obtaining a solution $v$
for some $v \in C^0([0,T); H^\sigma(\R^n))$ we have to show $\m ={\bf \hat{e}}_3 + v$ is $\St$ valued. For this purpose
we consider the function $ \phi(\m) = (1-|\m|^2)^2$
and observe $\phi(\m) \le  2 |v|^2 +|v|^4$ which is integrable in space for all $t \in (0,T)$. A straight forward calculation
and integration by parts show
$$ \frac{d}{dt} \int_{\R^n} \phi(\m(t)) \, dx \le 4 \|\nabla \m(t)\|_{L^\infty}^2   \int_{\R^n} \phi(\m(t)) \, dx$$
for all $t \in (0,T)$, and we deduce from Gronwall's lemma that $\m$ is  indeed $\St$ valued.

\subsection{Higher order energy estimates}
The proof of persistence given an $L^\infty$ bound on the gradient,
is based on the following a priori estimate which can be seen as a higher order version of \eqref{eq:energy_conservation}.

\begin{lemma}
Given any integers $\sigma \ge 2$ there exists a constant $c>0$ that only depends on $\lambda$ and $\sigma$
with the following property:
If $\m  \in  C^0\left((0,T); H^{\infty}_\ast (\R^n)\right)$ is a solution of \eqref{eq:LLG}, then 
\begin{equation}\label{eq:generalized_energy}
\frac{d}{dt} \,  \|\nabla \m(t)\|_{H^{\sigma-1}}^2 +  \lambda \norm{ \nabla \m(t)}{H^\sigma}^2  \le 
  c  \left(1+\| \nabla \m(t)\|_{L^{\infty}}^2\right) \|\nabla \m(t)\|_{H^{\sigma-1}}^2
  \end{equation}
 hence with $\displaystyle{C(t) = c \int_0^t \left(1+\|\nabla \m(s)\|_{L^\infty}^2\right) \, ds}$ we have for all $t \in (0,T)$
\begin{equation}\label{eq:Gronwall}
\|\nabla \m(t)\|_{H^{\sigma-1}}^2  \le e^{C(t)}\; 
\|\nabla \m(0)\|_{H^{\sigma-1}}^2.
\end{equation}
\end{lemma}
 
 \begin{proof}
Given any multi-index $1 \le |\alpha|\le \sigma-1$ we have
\begin{equation*}
\del^\alpha(\m \times \nabla \m) = \m \times \del^\alpha \nabla \m + R
\quad \text{where} \quad \|R\|_{L^2} \le c \; \|\nabla \m\|_{L^\infty}  \|\nabla \m\|_{H^{\sigma-1}}.
\end{equation*} 
In fact, the commutator term $R$ only contains mixed derivatives and no absolute term $\m$, and we can use
the following interpolation inequality (see \cite{Taylor_3}, p.9)
$$ \| \del^\beta f \;  \del^\gamma g\|_{L^2} \le c \; \norm{f}{L^\infty} \norm{g}{H^{\ell}}+ \norm{g}{L^\infty} \norm{f}{H^{\ell}}
\quad \text{for} \quad |\beta|+|\gamma|=\ell. $$
Moreover, we obtain by a similar argument with a generic constant $c$
\begin{eqnarray*}
\| \del^\alpha \left( |\nabla \m|^2 \m \right)  \|_{L^2} &\le& c \, \| \m\|_{L^\infty} 
 \|\nabla \m\|_{L^\infty} \|\nabla \m\|_{H^{\sigma-1}} +  \|\nabla \m\|_{L^\infty}^2 \|\nabla \m\|_{H^{\sigma-2}} \\ \nonumber
&\le& c \; \left(1+\|\nabla \m\|_{L^\infty}^2\right) \|\nabla \m\|_{H^{\sigma-1}}.
\end{eqnarray*} 
Applying $\del^\alpha$ to \eqref{eq:LLG} and multiplying it by $\del^\alpha \m$ we obtain
upon integration by parts
\begin{equation*}
\begin{split}
\frac{d}{dt} \|\del^\alpha \m\|_{L^2}^2 + 2 \lambda \norm{\del^\alpha \nabla \m}{L^2}^2 \le
  c \, \|\nabla \m\|_{L^\infty} \|\nabla \m\|_{H^{\sigma-1}} \norm{\del^\alpha \nabla \m}{L^2} \\
   +  c   \left(1+\|\nabla \m\|_{L^\infty}^2\right) 
  \|\nabla \m\|_{H^{\sigma-1}}^2.
\end{split}   
\end{equation*}
Then \eqref{eq:generalized_energy} follows from Young's inequality and after summing over all $1 \le |\alpha|  \le \sigma-1$, 
while
\eqref{eq:Gronwall} follows subsequently from Gronwall's inequality.
\end{proof}

A modified version of the Lemma holds true for approximate solutions obtained in a suitable contruction process.  
For $\sigma$ as in Proposition \ref{prop:local_classic} the estimate provides, by virtue of Sobolev 
embedding $H^{\sigma-1}(\R^n) \hookrightarrow L^\infty(\R^n)$, the requisite a priori estimate for
proving local existence. The bound on $\int_0^T \norm{ \nabla \m(t)}{H^\sigma}^2\, dt$ obtained for \eqref{eq:generalized_energy} upon integration in time and 
the uniqueness result below can be used to bootstrap higher regularity estimates in the interior. We refer to \cite{Taylor_pseudo,Taylor_3} for details.

%
%
%

\subsection{Uniqueness and stability}

\begin{lemma}\label{lemma:stability}
There exists a constant $c$ with the following property:
If $\m_1$ and $\m_2$ are classical solutions as in Proposition \ref{prop:local_classic}
on a common time interval $[0,T)$, then 
$$\norm{(\m_1 - \m_2)(t)}{L^2}^2 
       \le \exp \left( c  \sum_{i=1,2}\int_0^t \|\nabla \m_i(s)\|_{L^\infty}^2 \, ds \right)
        \norm{(\m_1 - \m_2)(0)}{L^2}^2$$
for every $t \in (0,T)$.
In particular, solutions in Proposition \ref{prop:local_classic} are unique. 
If $n \ge 3$ there exists a constant $\eta>0$ with the following property:
If $\m_1$ and $\m_2$
are weak solutions in the sense our Definition
in the space $L^\infty((0,T);\dot{W}^{1,n}(\R^n))$ with
$$\sup_{t \in (0,T)} \|\nabla \m_i(t)\|_{L^n} < \eta$$ 
for $i=1,2$, then 
$$\norm{(\m_1- \m_2)(t)}{L^2}^2 + \frac{\lambda}{2} \int_0^t \norm{\nabla ( \m_1- \m_2)(s)}{L^2}^2 \, ds  \le
  \norm{(\m_1 - \m_2)(0)}{L^2}^2$$
for every $t \in (0,T)$.   In particular, weak solutions in this class are unique.
 \end{lemma}

\begin{proof}
Observe that by \eqref{eq:LLG2} and with
the notation $\Phi =\m_1-\m_2$ we have
\begin{equation*}
\begin{split}
 \Phi_t = \lambda \left(  \Delta \Phi + |\nabla \m_1|^2 \Phi +  \m_2  \nabla \Phi : \nabla (\m_1+ \m_2) \right)
  - \nabla \cdot( \Phi \times \nabla \m_1 + \m_2 \times \nabla \Phi).
\end{split}
\end{equation*}
Using $2  \Phi$  as a test function yields, for every $t \in (0,T)$
\begin{equation*}
\begin{split}
 \|\Phi (t) \|^2_{L^2}  + 2 \lambda  \int_0^t  \|\nabla \Phi \|_{L^2}^2 \, ds 
 \le  \|\Phi (0) \|^2_{L^2}  + 2\lambda \int_0^t \| \Phi \nabla \m_1\|^2_{L^2}  ds \\
                                                                                 +2 \int_0^t   \Big( \lambda \, \|\Phi \, \nabla(\m_1 + \m_2)\|_{L^2}
                                                                                  +  \|\Phi \, \times \nabla \m_1\|_{L^2}  \Big)\| \nabla \Phi \|_{L^2} \, ds
\end{split}    
\end{equation*}
Now the first claim follows easily from Young's and Gronwall's inequality.\\
For the second claim we also use the Sobolev inequality $\|\Phi\|_{L^\frac{2n}{n-2}} \le c  \|\nabla \Phi\|_{L^2}$ in order 
to replace the $L^\infty$ bounds by $L^n$ bounds on $\nabla \m_i$  leading to
\begin{equation*}
\begin{split}
 \|\Phi (t) \|^2_{L^2}  + \lambda  \int_0^t  \|\nabla \Phi \|_{L^2}^2 \, ds \le   \|\Phi (t) \|^2_{L^2}   +  ( \eta + \eta^2) \, c
\int_0^t  \|\nabla \Phi\|^2_{L^2}\, ds \end{split}    
\end{equation*}
for some constant $c$ that only depends on $\lambda$ and $n$.
Choosing $\eta$ sufficiently small, the second term on the right can be adsorbed, and the claim follows.
 \end{proof}

The Lemma is useful in approximating local solutions, and for the time being we shall assume 
$ \m_0 \in H^\infty_\ast(\R^n; \St).$
Our aim is to show that classical solutions $\m \in C^0([0,T); H^\infty_\ast(\R^n))$ are
global provided $\nabla \m_0$ is sufficiently small in the $L^n$ norm. 
A well-konwn approximation device due to Schoen and Uhlenbeck enables us to extend this result to weak initial data.

\section{Derivation of the covariant Landau-Lifshitz-Gilbert system}
The formulation of harmonic flows using orthonormal frames on the tangent bundle of the target manifold under
a suitable gauge has been successfully used in the context of wave maps (see e.g. \cite{Freire_M_S:97,Shatah-Struwe:02})
and more recently in the context of Schr\"odinger maps (see e.g. \cite{Nahmod:03,Tao:04, Bejenaru:07, Bejenaru:08, Nahmod:09}).
The construction of what we shall call the covariant Landau-Lifshitz-Gilbert system is extrinsic and analog to the case of Schr\"odinger maps into $\St$.
We shall keep our presentation brief and particularly refer to the self-contained presentation in \cite{Bejenaru:07}, Section 2.
We shall see that in these new coordinates the derivative function will solve a complex Ginzburg-Landau
equation. 

\begin{notation} 
Latin indices $k,\ell \in \{1, \dots, n\}$ will be used for spatial components.  
$\nabla$ and $\Delta$ will denote the spatial gradient and Laplace operator, respectively. For $a \in \R^n$ and $u \in \C^n$ we set 
$a \cdot u = \sum_{k=1}^n a_k u_k \in \C$, and
accordingly $\div u= \nabla \cdot u$ will denote the euclidean divergence. Greek indices $\alpha, \beta \in \{0, \dots, n\}$ will be used for space-time components, where $\alpha=0$ is the time index. Complex space-time vectors are denoted by bold letters ${\bf u}=(u_0,u) \in \C^{1+n}$.
Finally we set $\del_\alpha = \dd{}{x_\alpha}$. 
\end{notation}

\subsection{Moving frames} \label{sec:moving}
We consider
$$ \m \in C^0\left([0,T]; H^\infty _\ast(\R^n)\right) \quad \text{and} \quad
      \dd{\m}{t} \in C^0\left([0,T]; H^{\infty}(\R^n)\right).$$
By virtue of a topological construction and a regularization argument, it has been shown in \cite{Bejenaru:07} \S 2 that there exist 
smooth tangent vector fields along $\m$
$$ X,Y :  (0,T) \times \R^n \to  T_{\m} \St$$
such that 
$$ |X|=|Y|=1\quad \text{and} \quad \m = X \times Y,$$ 
that is, the pair $\{X,Y\}$ forms an orthonormal tangent frame along $\m$.
Moreover, the  vector fields $X$ and $Y$ inherit the regularity properties of $\m$, i.e.
$$ \del_\alpha X,   \del_\alpha Y 
 \in C^0\left([0,T]; H^{\infty}(\R^n)\right) \quad \text{for} \quad \alpha \in \{0, \dots, n\}$$
We define the map
$$ {\bf a}=(a_0,a) \in C^0\left([0,T]; H^\infty( \R^n; \R^{1+n}) \right)$$
with coefficients $a_\alpha$ defined by
$$ a_\alpha = \scp{\del_\alpha X}{Y}= - \scp{\del_\alpha Y}{X},$$
where $\langle\cdot, \cdot \rangle$ is the euclidean product on $\R^3$, giving rise to a unitary connection
$$D_\alpha = \del_\alpha+i a_\alpha. $$
Representing the space-time gradient of $\m$ in terms of $X+iY$ we obtain coefficients
\begin{equation}\label{eq:u_comp}
 u_\alpha = \scp{\del_\alpha \m}{X} + i \scp{\del_\alpha \m}{Y}
 \end{equation}
such that
$$\del_\alpha \m= \Re (u_\alpha) \, X + \Im (u_\alpha) \, Y$$
The coefficients $u_\alpha$ form a smooth complex vector function 
$${\bf u}=(u_0,u) \in C^0\left( [0,T];H^\infty( \R^n; \C^{1+n})\right).$$
Taking into account the relations
$$\del_\alpha X  =  -\Re (u_\alpha) \, \m + a_\alpha \,  Y \quad \text{and} \quad \del_\alpha Y = - \Im (u_\alpha) \, \m  - a_\alpha \, X $$
we obtain the \textit{zero torsion identity}  
\begin{equation} \label{eq:zero_torsion}
 D_\alpha u_\beta = D_\beta u_\alpha
 \end{equation}
and the \textit{curvature identity}
\begin{equation}\label{eq:curvature} R_{\alpha \beta} :=   [D_\alpha, D_\beta] 
= i \left( \del_\alpha a_\beta -\del_\beta a_\alpha \right)
=  i \, \Im (  u_\alpha \bar{u}_\beta). 
\end{equation}
\subsection{The covariant LLG system}
Suppose  $T \in (0,T_\ast)$ and 
$$\m \in C^0([0,T_\ast); H^\infty_\ast(\R^n;\St))$$ 
is a solution of LLG with  
$\m(0)=\m_0 \in H^\infty_\ast(\R^n;\St)$
as in Proposition \ref{prop:local_classic}. On $[0,T]$ the construction of a moving frame $\{X,Y\}$ along $\m$ 
from Section \ref{sec:moving} can be carried out. A straight forward calculation shows
$$ \Delta \m = \sum_{k=1}^n\left( \del_k \Re(u_k) - a_k\, \Im(u_k) \right)X +\sum_{k=1}^n \left( \del_k \Im(u_k) + a_k\,  \Re(u_k) \right)Y + |a|^2 \m.$$
Upon multiplication by $X+iY$,  the Landau-Lifshitz-Gilbert equation 
$$  \dd{\m}{t}  = \lambda  \left( \Delta \m + |\nabla \m|^2 \m \right) -\m \times \Delta \m $$
turns into the following system that we shall call the convariant LLG system.

\begin{proposition}\label{prop:covariant}
Suppose $\m$ is a solution of the Landau-Lifshitz-Gilbert equation as in Proposition \ref{prop:local_classic} on an interval $(0,T_\ast)$
with initial data $\m_0 \in H^\infty_\ast(\R^n;\St)$. For $T \in (0,T_\ast)$ and with the notation from Section \ref{sec:moving} 
$$({\bf u},{\bf a}) \in C^0 \left([0,T]; H^\infty( \R^n;\C^{1+n} \times \R^{1+n})\right) $$ solves the
covariant Landau-Lifshitz-Gilbert system:
\begin{equation} \label{eq:covLLG}
\begin{cases}
u_0   = (\lambda - i) \sum_{k=1}^n D_k u_k\\
D_\beta u_\alpha   = D_\alpha u_\beta \\
 \del_\alpha a_\beta - \del_\beta a_\alpha     =    \Im (  u_\alpha \bar{u}_\beta) 
\end{cases} \end{equation}
Hence $u=(u_1,\dots, u_n)$ solves the covariant complex Ginzburg-Landau equation
\begin{equation}\label{eq:covLLG2}
D_0 u_\ell =   (\lambda - i) \sum_{k=1}^n \left( D_k D_k u_\ell + R_{\ell k}  u_k \right)  \quad \text{for} \quad \ell \in \{1, \dots, n\} 
\end{equation}
and attains initial values $u(0) = \langle \nabla \m_0,X \rangle +  i \langle \nabla \m_0,Y \rangle$ in $H^{\infty}(\R^n;\C^n)$.
\end{proposition}

 \subsection{The Coulomb gauge}  
We observe that the covariant LLG system in invariant with respect to the individual choice
of the moving orthonormal frame $\{X,Y\}$. This local rotation invariance translates into the fact that 
 \eqref{eq:covLLG} is invariant with respect to gauge transformations
$$ u \mapsto e^{-i \theta} u \quad \text{and} \quad a_\alpha \mapsto a_\alpha + \del_\alpha \theta.$$
A canonical choice is the Coulomb gauge which is characterized by the equation
$$ \div a^\ast = \sum_{\ell=1}^n\del_\ell a_\ell^\ast =0$$ 
and obtained from our initial moving frame $\{X,Y\}$ by a gauge change represented by a suitable gradient
$$\nabla \theta \in  C^0([0,T]; H^{\infty}(\R^n))$$ 
where $\theta$ solves, for all $t \in [0,T]$, an elliptic equation
$$-\Delta \theta(t) = \div a(t) \quad \text{in} \quad \R^n .$$
Note that $\nabla \theta(t)$ can be obtained by taking double Riesz transforms
of $a$ which are bounded operators on every Sobolev space $H^{\sigma}(\R^n;\R^n)$, see e.g. \cite{Bergh,Iwaniec_Martin}.
It follows that the new frame $\{X^\ast, Y^\ast\}$ inherits the regularity of $\{X, Y\}$, same for ${\bf a}^\ast$ and ${\bf u}^\ast$,
see \cite{Bejenaru:07} Proposition 2.3. In the sequel we will skip the superscript $\ast$.
\subsection{$L^p$ estimates for ${\bf a}(t)$}
Recall from \eqref{eq:curvature} that for $\alpha, \beta \in \{ 0, \dots, n \}$  
$$ \del_\alpha a_\beta  - \del_\beta a_\alpha     =    \Im (  u_\alpha \bar{u}_\beta) $$
which implies that under the Coulomb gauge  for $\alpha = 0, \dots, n$  
\begin{equation}\label{eq:Coulomb_pot}
-\Delta  a_\alpha =    \del_k \Im (  u_\alpha \bar{u}_k) 
\end{equation}
 
A priori estimates in Section \ref{Sec:apriori} take into account higher power
$L^p$ norms of $u(t)$, i.e. $p=n/\delta$ larger than the critical power $n$, and $L^n$ norms
of $\nabla u(t)$. In turn we will need corresponding estimates for
$a(t)$ in terms of $u(t)$.

\begin{lemma}\label{lemma:a_L^p}
Suppose $\delta \in (\frac{1}{2},1)$ and $n\ge 3$. Then we have under the Coulomb gauge with
a constant $c$  that only depends on $\delta$, $\lambda$ and $n$
\begin{equation}\label{eq:a_x_L^p}
\norm{a(t)}{L^\frac{n}{2\delta-1}} \le c \; \|u(t)\|_{L^\frac{n}{\delta}}^2
\end{equation}  
for all $t \in [0,T]$.
Moreover, there exists a decomposition $a_0=a_0^{(1)}+a_0^{(2)}$ where
$$a_0^{(1)} \in C^0([0,T];L^\frac{n}{\delta}(\R^n)) \quad \text{and} \quad  
a_0^{(2)} \in C^0([0,T];L^\frac{n}{4 \delta - 2}(\R^n))$$
and such that
\begin{equation}\label{eq:a_0_1_L^p}
\|a_0^{(1)}(t)\|_{L^\frac{n}{\delta}}\le c \;  \|u(t)\|_{L^\frac{n}{\delta}} \|\nabla u(t)\|_{L^n}
\end{equation}
and
\begin{equation} \label{eq:a_0_2_L^p}
\| a_0^{(2)}(t)\|_{L^\frac{n}{4 \delta - 2}} \le c \; \|u(t)\|_{L^\frac{n}{\delta}}^2
\end{equation}
for all $t \in [0,T]$.
\end{lemma}

\begin{proof} All estimates are of elliptic type, and we shall suppress the time dependence.
Recall that if $1<p<n$ and $f \in L^p(\R^n;\R^n)$ then
$$-\Delta v =   \div f \quad \text{in} \quad \R^n$$
has a (weak) solutions $v \in \dot{W}^{1,p}(\R^n)$ given by
$$v=(-\Delta)^{-\frac{1}{2}} \sum_{k=1}^n \mathcal{R}_k f_k$$ 
where the $\mathcal{R}_k$ are the
Riesz transforms and $(-\Delta)^{-\frac{1}{2}}$ is the Riesz potential corresponding
to the Fourier multiplier $1/|\xi|$.
It follows from the Hardy-Littlewood-Sobolev inequality, see e.g. \cite{Stein} p.119, that
$v \in L^\frac{np}{n-p}(\R^n)$ with
\begin{equation}\label{eq:hls}
\| v \|_{ L^\frac{np}{n-p}} \le c \,\| f\|_{L^p}.
\end{equation}
Moreover, the solution $v$ is unique in its class.
Hence, in view of \eqref{eq:Coulomb_pot}, 
$$- \Delta u_\ell =   \sum_{k=1}^n \del_k \Im (  \bar{u}_k u_\ell) =  \div \Im (   \bar{u} u_\ell) $$
and estimate \eqref{eq:a_x_L^p}  follows from \eqref{eq:hls} with $p=\frac{n}{2\delta}$.
For the decomposition of $a_0$ we observe that by Proposition \ref{prop:covariant}
\begin{eqnarray*}
 \Im ( \bar{u} u_0) &=& \Im  \Big( (\lambda - i) \,  \bar{u}  \,  D_\ell u_\ell\; \Big)\\ 
&=& \Im \Big( (\lambda - i) \, \bar{u} \,  \div u    \Big)
         + \Re \Big( (\lambda - i) \,  (a \cdot u) \,  \bar{u} \Big).
\end{eqnarray*}
Accordingly we define $a_0^1$ and $a_0^2$ by the equations
$$- \Delta a_0^{(1)} = \div \Big( \lambda\,  \Im \big( \bar{u} \, \div u  \big) -  \Re \big( \bar{u}\, \div u  \big) \Big)$$
and 
$$- \Delta a_0^{(2)} =  \div \Big(   \lambda \, \Re \big(  \bar{u} \, (a\cdot u)  \big)   + \Im \big(\bar{u} \, (a\cdot u)  \big) \Big).$$
Estimate \eqref{eq:a_0_1_L^p} follows again from H\"older and 
\eqref{eq:hls} where $p=\frac{n}{1+\delta}$. 
Taking into account \eqref{eq:a_x_L^p}, estimate \eqref{eq:a_0_2_L^p} follows with $p=\frac{n}{4 \delta -1}$.
The proof is complete.
\end{proof}

\begin{remark} For Schr\"odinger maps, i.e. for $\lambda =0$,  we would obtain in view of \eqref{eq:zero_torsion}
$$ -\Delta a_0 = \div \div \left( \frac{1}{2} |u|^2 \; {\rm \bf id} - \Re(\bar{u} \otimes u) \right)$$ 
and hence $\| a_0 \|_{L^p} \le c \, \|u\|_{L^{2p}}$ for all $p \in (1, \infty)$ by the Cald\'eron-Zygmund inequality.
The fact in case of LLG the estimate for $a_0$ takes into account $\nabla u$ is not essential for our regularity
argument.
\end{remark}

\section{Estimates for the covariant Landau-Lifshitz-Gilbert system}       
\label{Sec:apriori}
In this section we consider 
$$ u\in C^0 \left([0,T]; H^\infty( \R^n;\C^{n})\right) $$
solving the covariant Ginzburg-Landau system \eqref{eq:covLLG2} under the Coulomb gauge.

\subsection{Linear estimates for the semigroup}
Let us first recall the fundamental estimates for the dissipative Schr\"odinger semigroup $S=S(t)$
which is generated by $(\lambda-i) \Delta$.
With a slight abuse of notation, we represent this semigroup in terms of the Fourier multiplier 
$$ \hat{S}_{t}(\xi) = e^{(i-\lambda) |\xi|^2t}$$ 
which is a Schwartz function for every $\lambda>0$ and $t>0$. Thus the associated kernel $S_t$ is also a Schwartz function
with the scaling property
$$S_t(x)=t^{-n/2} S(x/\sqrt{t}).$$
Writing $S(t) f = S_t \ast f$ for $f \in L^p(\R^n;\C)$ we obtain from Young's inequality
the following mapping properties on Lebesgue spaces parallel to
those of the heat kernel (see also \cite{Wu:99} Proposition 2.3):
\begin{lemma}\label{lemma:semigroup} Suppose $\lambda>0$, 
$1 \le p \le q \le \infty$ and $ \sigma\in \N$. Then there exists a constant $c>0$ such that 
$$ \| \nabla^\sigma S(t) \|_{\mathcal{L}(L^p;L^{q})} \le  c \, t^{-\frac{n}{2}\left( \frac{1}{p}-\frac{1}{q} \right)-\frac{\sigma}{2}} \quad \text{for all} \quad t>0.$$
\end{lemma}
\subsection{Nonlinear estimates}
We recall that the covariant LLG system in Proposition \ref{prop:covariant}
gives rise to a complex Ginzburg-Landau system \eqref{eq:covLLG2} which, under
the Coulomb gauge, can be written as 
\begin{equation}\label{eq:NLS}
\dd{u}{t} = (\lambda - i)  \Delta u +F({\bf a},u) 
\end{equation} 
where, for ${\bf a}=(a_0,a)$, the nonlinearity is given by 
\begin{equation}\label{eq:nonlinearity}
 F_\ell({\bf a},u) = 
 (\lambda - i) \left\{i\,  \sum_{k=1}^n \Big( \Im(u_\ell \, \bar{u}_k) \, u_k  \Big)+ 2 \, i \,  (a \cdot \nabla) u_\ell  - |a|^2  u_\ell \right\} -i \, a_0 \, u_\ell.
 \end{equation}
Taking into account the decomposition of $a_0=a_0^{(1)}+a_0^{(2)}$ from Lemma \ref{lemma:a_L^p} we find that $F$ splits into five terms
$$  F({\bf a},u)= F^{(1)}+   \dots +  F^{(5)}.$$
Apart from the cubic term $F^{(1)}$ that only depends on $u$ in a local fashion, the functions
$F^{(2)}$ to $F^{(5)}$ can be considered as nonlocal multilinear operators
acting on $u$, that is ${\bf a}={\bf a}(u)$ is considered as a function of $u$. More precisely, in view of Lemma \ref{lemma:a_L^p},
$F^{(2)}$ and $F^{(4)}$ are quadratic in $u$ and linear
in $\nabla u$, while $F^{(3)}$ and  $F^{(5)}$ are quintic in $u$.
Therefore we introduce the functions
\begin{eqnarray*}
f^{(1)}:=F^{(1)} &=& (\lambda  - i) \, i \, \sum_{k=1}^n \Im (u \, \bar{u}_k)   \, u_k  ,\\
f^{(2)}:=F^{(2)}+ F^{(4)} &=&  (\lambda - i )   \; 2 i  \,
 (a \cdot \nabla )u -  i \, a_0^{(1)} \, u,\\
f^{(3)}:=F^{(3)}+ F^{(5)} &=& -( \lambda-i) \; |a|^2  u -i \, a_0^{(2)} \, u
\end{eqnarray*}
 and obtain from Lemma \ref{lemma:a_L^p} and H\"older's inequality:
%
\begin{lemma}\label{lemma:N_estimates}
Suppose $n \ge 3$ and $\delta \in (\frac{1}{2},1)$. Then there exists a constant $c$ such that  
\begin{eqnarray} \label{eq:f_1_estimate}
  \|f^{(1)}(t)\|_{L^\frac{n}{3\delta}}
  & \le & c \; \|u(t)\|_{L^\frac{n}{\delta}}^3,
\\ \label{eq:f_2_estimate}
    \| f^{(2)}(t)  \|_{L^\frac{n}{2 \delta}}& \le& c \; \|u(t)\|_{L^\frac{n}{\delta}}^2 \| \nabla u(t)\|_{L^n},
\\ \label{eq:f_3_estimate}
  \| f^{(3)}(t) \|_{L^\frac{n}{5 \delta-2}} & \le& c \; \|u(t)\|_{L^\frac{n}{\delta}}^5
\end{eqnarray}      
for all $t \in [0,T]$.
\end{lemma}
%
%
\subsection{Duhamel's principle and Fujita-Kato type estimate} At this point we have essentially
reduced the problem to a linear one.
Writing $f=f^{(1)} +f^{(2)}+ f^{(3)}$ we obtain the following representation formula
\begin{equation}\label{eq:Duhamel}
u(t) = S(t) u(0) +  S  \ast f (t) 
\end{equation}
where $S(t) f = S_t \ast f$ (i.e. convolution in space) and 
\begin{equation}\label{eq:G}
(S \ast f )(t) := \int_0^t S(t-s) f(s) \, ds
\end{equation}
i.e. convolution in space and time.
Based on this representation we shall derive a priori estimates in suitable (scaling invariant) 
weighted-in-time Lebesgue-Sobolev spaces. For this purpose we set, for $\delta \in (\frac{1}{2},1)$ and $t \in [0,T]$,
$$ K(t)= \sup_{\tau \in (0,t)} \tau^\frac{1-\delta}{2} \| u(\tau)\|_{L^\frac{n}{\delta}}$$
$$ K'(t)= \sup_{\tau \in (0,t)} \tau^\frac{1}{2} \| \nabla u(\tau)\|_{L^{n}}$$ 
and 
$$R(t)=\max \{ K(t),K'(t) \}.$$ It follows from Sobolev embedding that
\begin{equation}\label{eq:Rto0}
R(t) \le c \, t^\frac{1-\delta}{2} \|u(t)\|_{H^\frac{n}{2}} \quad  \text{for all} \quad t \in [0,T].
\end{equation}
Hence $R:(0,T) \to [0,\infty)$ is continuous with
$\displaystyle{\lim_{t \searrow 0} R(t) = 0}$. We also introduce
$$R_0(t)= \max \left\{ \sup_{\tau \in (0,t)}
                     \tau^\frac{1-\delta}{2} \| S(\tau) u(0)\|_{L^\frac{n}{\delta}} \;,\;
                      \sup_{\tau \in (0,t)} \tau^\frac{1}{2} \| \nabla S(\tau) u(0)\|_{L^{n}} \right\}.$$ 
According to Lemma \ref{lemma:semigroup} we have
\begin{equation}\label{eq:R_0}
R_0(t) \le c \, \|u(0)\|_{L^n}
\end{equation}
independently of $t>0$ and for a constant $c$ that only depends on $\delta$, $\lambda$ and $n$.                       

\begin{lemma} \label{lemma:R_estimate}
Suppose $\delta \in (\frac{3}{5},\frac{2}{3})$. Then there exists a positive constant $c_0$
such that
\begin{equation} \label{eq:R_estimate}
 R(t) \le R_0(t) + c_0 \left( R(t)^3+R(t)^5 \right).
 \end{equation}
Moreover, there exits a positive constant $r_0$ with the following property: If 
$$\sup_{t \in (0,T)} R_0(t) < r_0$$ then $R(t) \le 2 \, R_0(t)$
for every $t \in [0,T]$.
\end{lemma}

\begin{proof}
Estimate \eqref{eq:R_estimate} will follow from \eqref{eq:Duhamel}.
Taking into account that for $\alpha,\beta<1$ (restricting the range of admissible $\delta$)
and for positive $t$ with a constant $c=c(\alpha,\beta)$
$$ \int_0^t (t-s)^{-\alpha} s^{-\beta} \, ds  = c \, t^{1-\alpha-\beta}$$
we obtain form Lemma \ref{lemma:semigroup} and Lemma \ref{lemma:N_estimates} with a generic constant $c$
$$ \left\|\left(S \ast f^{(1)} \right) (t) \right\|_{L^\frac{n}{\delta}} \le c \; K(t)^3 \int_0^t (t-s)^{-\delta} s^{-\frac{3}{2}(1-\delta)} \; ds
= c \, K(t)^3\;  t^{\frac{\delta-1}{2}},$$

$$ \left\| \left(S \ast f^{(2)} \right)(t) \right\|_{L^\frac{n}{\delta}} \le c \; K(t)^2 K'(t)\int_0^t (t-s)^{-\frac{\delta}{2}} s^{-\frac{3}{2}+\delta}     \; ds
 = c \, K(t)^2 K'(t) \; t^{\frac{\delta-1}{2}},$$

$$ \left\| \left(S \ast f^{(3)} \right)(t)\right\|_{L^\frac{n}{\delta}} \le c \; K(t)^5  \int_0^t (t-s)^{1 - 2 \delta} s^{-\frac{5}{2}(1-\delta)} \; ds=
 c \, K(t)^5 \; t^{\frac{\delta-1}{2}}$$

for $t \in (0,T)$, provided $\delta \in (\frac{3}{5},1)$ (used in the 3rd estimate). Moreover
$$ \left\| \nabla \left(S \ast f^{(1)}  \right)(t) \right\|_{L^n} \le c \; K(t)^3 \int_0^t (t-s)^{-\frac{3\delta}{2}}
 s^{-\frac{3}{2}(1-\delta)} \; ds = c \, K(t)^3\;  t^{-\frac{1}{2}},$$

$$ \left\| \nabla \left(S \ast f^{(2)} \right)(t) \right\|_{L^n} \le c \; K(t)^2 K'(t)\int_0^t (t-s)^{-\delta} s^{-\frac{3}{2}+\delta}     \; ds
 = c \, K(t)^2 K'(t) \; t^{-\frac{1}{2}},$$

$$ \left\| \nabla \left(S \ast f^{(3)}  \right) (t)\right\|_{L^n} \le c \; 
   K(t)^5  \int_0^t (t-s)^{-\frac{1}{2} (5 \delta -2)} s^{-\frac{5}{2}(1-\delta)} \; ds = c \, K(t)^5 \; t^{-\frac{1}{2}}$$

for all $t \in (0,T)$ provided $\delta \in (\frac{1}{2},\frac{2}{3})$ (used in the 1st estimate). Hence \eqref{eq:R_estimate} follows.

\medskip

To prove the second claim, we argue by contradiction and assume there exists $t_0 \in (0,T)$ such that $R(t_0) = 2 R_0(t_0) \not=0$. 
Then, by virtue of \eqref{eq:R_estimate},
$$ 2 R_0(t) \le R_0(t_0) +   c_0 \, \left( 8 \, R_0(t_0)^2+ 32 \, R_0(t_0)^4 \right) R_0(t_0).$$
Thus for $0<R_0(t_0) < r_0$ with $r_0$ sufficiently small
$$  1/c_0 \le 8 \, R_0(t_0)^2+ 32 \, R_0(t_0)^4 < 1/c_0,$$
a contradiction. Since $R$ is continuous with $\lim_{t \searrow 0} R(t)=0$, the claim follows.\end{proof}

Recall that by \eqref{eq:R_0} the smallness condition can be expressed in terms of $\|u(0)\|_{L^n}$. Hence we obtain:

\begin{proposition}  \label{prop:R_estimate} Suppose $\delta \in (\frac{3}{5},\frac{2}{3})$. 
Then there exist positive constants $\rho$ and $c$ depending only on $\delta$, $\lambda$ and $n$ with the following property: If
$$\| u(0) \|_{L^n} < \rho$$
then the following estimate holds true for every $t \in (0,T]$
$$  t^\frac{1-\delta}{2} \| u(t)\|_{L^\frac{n}{\delta}} + t^\frac{1}{2} \| \nabla u(t)\|_{L^{n}} 
 \le c \, \| u(0) \|_{L^n}.$$ 
\end{proposition}

\subsection{Uniform bounds in $L^n$ and H\"older continuity} 

\begin{lemma}
Suppose $\alpha \in (0,\frac{1}{5})$. Then there exits a constant $\rho$ and $c$ depending only on
$\alpha$, $\lambda$ and $n$ with the following property: If
$$ \| u(0) \|_{L^n} < \rho $$
then the following estimate holds every $t \in (0,T]$
$$  t^{\frac{1+\alpha}{2}}\|\nabla u(t)\|_{L^\frac{n}{1-\alpha}}  \le c  \, \|u(0)\|_{L^n} .$$
\end{lemma}

\begin{proof} 
We first show that for some constant $c>0$ independent of $t \in (0,T]$ \begin{equation}\label{eq:M_estimate}
 t^{\frac{1+\alpha}{2}}\|\nabla u(t)\|_{L^\frac{n}{1-\alpha}} \le c \, \| u(0) \|_{L^n} + c\,\left( R(t)^3+R(t)^5 \right). 
 \end{equation}
To this end we repeat the argument in Lemma \ref{lemma:R_estimate} with a change in the 
exponent  of $t-s$ according to Lemma \ref{lemma:semigroup}. With a generic constant $c$ we obtain for all $t \in (0,T]$
$$ \left\| \nabla \left(S \ast f^{(1)} \right)(t) \right\|_{L^\frac{n}{1-\alpha}}\le c \; R(t)^3 \int_0^t (t-s)^{-\frac{1}{2}(3\delta+\alpha)} s^{-\frac{3}{2}(1-\delta)} \; ds
 = c \, R(t)^3\;  t^{-\frac{1+\alpha}{2}},$$
$$ \left\| \nabla \left(S \ast f^{(2)} \right)(t) \right\|_{L^\frac{n}{1-\alpha}} \le c \; R(t)^3\int_0^t (t-s)^{-\frac{1}{2} (2 \delta+\alpha)} s^{-\frac{3}{2}+\delta}     \; ds
 = c \, R(t)^3 \; t^{-\frac{1+\alpha}{2}},$$
 $$ \left\| \nabla \left(S \ast f^{(3)} \right)(t) \right\|_{L^\frac{n}{1-\alpha}} \le c \; 
   R(t)^5  \int_0^t (t-s)^{-\frac{1}{2} (5 \delta -2 +\alpha)} s^{-\frac{5}{2}(1-\delta)} \; ds = c \, R(t)^5 \; t^{-\frac{1+\alpha }{2}}.$$
The first estimate requires $3\delta+\alpha<2$ while the third estimate requires $ \delta>\frac{3}{5}$
which is possible by an appropriate choice of $\delta \in (\frac{3}{5},\frac{2}{3})$ provided $\alpha \in (0,\frac{1}{5})$.
Finally we observe that according to Lemma \ref{lemma:semigroup}
$$ t^{\frac{1+\alpha}{2}}  \| \nabla S(t)u(0) \|_{L^\frac{n}{1-\alpha}}  \le c \, \|u(0)\|_{L^n}$$
for every $t>0$ and with a constant $c$ that only depeneds on $\alpha$ and $n$.
Hence Duhamel's formula \eqref{eq:Duhamel} implies \eqref{eq:M_estimate}. If $\rho$ is sufficiently small
Lemma \ref{lemma:R_estimate} implies
$$R(t) \le 2 R_0(t) \quad \text{while by \eqref{eq:R_0} } \quad R_0(t) \le c \, \|u(0)\|_{L^n}< c \rho, $$
and we obtain 
$t^{\frac{1+\alpha}{2}}\|\nabla u(t)\|_{L^\frac{n}{1-\alpha}} \le c\,\left( 1 + \rho^2+\rho^4 \right) \| u(0) \|_{L^n}$ as claimed.
\end{proof}

The Lemma implies, by virtue of Morrey's inequality, H\"older continuity of $u(t)$ for every $t \in (0,T)$
with H\"older
exponent $\alpha \in (0, \frac{1}{5})$ and bounds that only depend on $\| u(0) \|_{L^n}$, $\alpha$, $\lambda $, $n$, and $t$.
Along the lines of the above argument we also obtain
\begin{equation} \label{eq:L^n_estimate}
\| u(t) \|_{L^n} \le c \; \| u(0) \|_{L^n} \quad \text{for all} \quad t \in [0,T]
\end{equation}
provided $\| u(0) \|_{L^n} < \rho$.
An $L^\infty$ bound for $u(t)$ can be obtained by scaling. In fact, 
 $$\norm{u(t)}{L^\infty} \le \norm{u(t)}{C^\alpha} \le c\left( \norm{u(t)}{L^n}+ \norm{\nabla u(t)}{L^\frac{n}{1-\alpha}}\right)$$
hence, see \cite{Taylor_3} p.9, 
$\displaystyle{\norm{u(t)}{L^\infty} \le c \;  \norm{u(t)}{L^n}^\frac{\alpha}{1+\alpha} \norm{\nabla u(t)}{L^\frac{n}{1-\alpha}}^\frac{1}{1+\alpha} 
 \le \frac{c}{\sqrt{t}} \|u(0)\|_{L^n}}$.

\begin{proposition} \label{prop:alpha} 
There exist constants $\rho$ and $c$ depending only on
$\lambda$ and $n$ with the following property: If
$$ \| u(0) \|_{L^n} < \rho $$
then  
the following estimate holds true for every $t \in [0,T]$
$$ \sqrt{t} \; \|u(t)\|_{L^\infty} + \| u(t) \|_{L^n} \le c \, \|u(0)\|_{L^n}.$$
\end{proposition}


  
 \section{Proof of the Theorem}  
 
 \subsection{Initial data in $H^\infty _\ast(\R^n;\St)$}\label{subsection:smooth} Given a map $\m_0$ in this class, there exist,
 according to Proposition \ref{prop:local_classic}, a terminal time $T_\ast>0$ and a smooth solution 
 $$\m  \in C^0([0,T_\ast); H^\infty_\ast(\R^n))$$
 with $\m(0)=\m_0$.
 For $T \in (0,T_\ast)$ we obtain, by Proposition \ref{prop:covariant}, a solution 
 $$u \in C^0([0,T]; H^\infty(\R^n;\C^n))$$
 of the corresponding covariant complex Ginzburg-Landau equation under the Coulomb gauge.
Since $|\nabla \m|=|u|$ we have for $t \in [0,T]$
 \begin{equation} \label{eq:grad_bound}
  \| \nabla \m(t) \|_{L^\infty} \le \|u(t)\|_{L^\infty}
 \quad \text{while} \quad 
  \|u(0)\|_{L^n} \le \| \nabla \m_0 \|_{L^n}. 
   \end{equation}

 Assuming  $\|\nabla \m_0\|_{L^n} <\rho$, 
 Proposition \ref{prop:alpha} applies, and we obtain 
 \begin{equation}\label{eq:L_estimate}
  \norm{\nabla\m (t)}{L^\infty} \le \frac{c}{\sqrt{t}} \|\nabla \m_0\|_{L^n}  \quad \text{for all} \quad t \in (0,T)
 \end{equation}
where $c$ only depends on $\lambda$ and $n$.
Since $T \in (0,T_\ast)$ is arbitrary we deduce from Proposition \ref{prop:local_classic} that $\m$ can be extended to all times. 
In turn, \eqref{eq:L_estimate} holds true for any $t>0$. In conjunction with the energy inequality \eqref{eq:energy_inequality} 
we obtain
\begin{equation}\label{eq:decay} 
 \| \m(t) - \m_\infty \|_{L^\infty} \le c \,  \|\nabla \m(t)\|_{L^{\infty}}^\frac{n-2}{2} \|\nabla \m(t) \|_{L^2}^\frac{2}{n} \le 
 c \,  t^{-\frac{n-2}{2n}} \, \|\nabla \m_0\|_{L^n \cap L^2} \end{equation}
 for all $t>0$.
Hence $\m(t) \to \m_\infty$ in $C^1(\R^n)$ as $t \to \infty$.


\subsection{Initial data in $H^\sigma_\ast(\R^n;\St)$ for $\sigma> \frac{n}{2}+1$}
Given a map $\m_0$ in this class we find, according to Proposition \ref{prop:local_classic}, a terminal time $T_\ast>0$ and
a local solution 
$$\m \in C^0([0,T_\ast); H^\sigma_\ast(\R^n)).$$ 
In order to obtain a uniform bound for $\nabla \m$,
we shall use the result in Section \ref{subsection:smooth} in conjunction with the following approximation result, originally due to Schoen and Uhlenbeck, see \cite{Schoen-Uhlenbeck:83} and \cite{Struwe:88} Proposition 7.2.

 \begin{lemma} \label{lemma:S_U}Suppose $\m: \R^n \to \St$ is such that $\m-\m_\infty \in H^1 \cap W^{1,n}(\R^n)$. Then
 there exists a sequence of maps $\upper{\m}{k} \in  H^\infty_\ast(\R^n;\St)$ such that  
 $$\m-\upper{\m}{k} \to 0 \quad \text{in} \quad  H^1 \cap W^{1,n}(\R^n).$$
 If $\m \in H^\sigma_\ast(\R^n)$ for some $\sigma >1$ then $\{\upper{\m}{k}\}$ is uniformly bounded in $H^\sigma_\ast(\R^n)$.
 \end{lemma}
 
 \begin{proof} With a standard mollifier $\varphi \in C^\infty_0(B_1(0))$ and $\varphi_\eps(x)=\eps^{-n} \varphi(x/\eps)$, we
 obtain, taking into account $\m - \m_\infty \in L^2(\R^n)$, 
 $$\varphi_\eps \ast \m - \m_\infty = \varphi_\eps \ast ( \m - \m_\infty)   \in H^\infty(\R^n).$$ 
Repeating the argument from \cite{Schoen-Uhlenbeck:83} yields for all $x \in \R^n$
\begin{equation}\label{eq:uniform_convergence}
\Big(1- |(\varphi_\eps \ast \m)(x)| \Big)^n= {\rm dist}\left( (\varphi_\eps \ast \m) (x),\St \right)^n \le c  \int_{B_\eps(x)} |\nabla \m|^n \, dx \to 0
\end{equation}
uniformly as $\eps \to 0$. For $\eps$ sufficiently small we have $\inf |\varphi_\eps \ast \m|>\frac{1}{2}$ and let 
 $$\upper{\m}{k}:=\frac{\varphi_\eps \ast \m}{|\varphi_\eps \ast \m|} \quad \text{for} \quad \eps =\eps_k \searrow 0.$$ 
We observe that  we can represent $\upper{\m}{k}-\m_\infty = F( \varphi_{\eps_k} \ast \m - \m_\infty)$ with a smooth function
$F:\R^3 \to \R^3$ from which we can assume $F(0)=0$ and which induces a bounded mapping on $L^\infty \cap H^\sigma(\R^n;\R^3)$
for every $\sigma \in \Z$,
see e.g. \cite{Taylor_3} p.11. Moreover,
$$ |\upper{\m}k-\m| \le |\upper{\m}k-\varphi_\eps \ast \m |+ |\varphi_{\eps_k} \ast \m - \m|  
\le 2\,  |\varphi_{\eps_k} \ast \m - \m|, $$
hence $\upper{\m}k-\m \to 0$ in $L^2(\R^n)$ as $k \to \infty$. Finally,
$ |\nabla \left( \varphi_{\eps} \ast \m \right) | \ge|\varphi_{\eps} \ast \m|\,  |\nabla \upper{\m}k|$
and therefore by \eqref{eq:uniform_convergence} we obtain, for $2 \le p \le n$, 
$$\limsup_{k \to \infty }  \| \nabla \upper{\m}k\|_{L^p} \le \lim_{k \to \infty} \|\varphi_{\eps_k} \ast \m \|_{L^p}
 = \|\m \|_{L^p} \le \liminf_{k \to \infty}  \| \nabla \upper{\m}k\|_{L^p}.$$
It follows that $ \upper{\m}k - \m \to 0$ in $H^1 \cap W^{1,n}(\R^n)$ as $k \to \infty$ as claimed.  
 \end{proof}

If we assume $\norm{\nabla \m_0}{L^n} < \rho$ 
with $\rho$ as in Proposition \ref{prop:alpha}, then,  for large $k$,  the same is true for a approximating sequence 
$\m^{(k)}_0$ of initial maps as in Lemma \ref{lemma:S_U}.
We obtain a corresponding sequence of global solutions $\m^{(k)} \in C^0([0,\infty);H^\infty_\ast(\R^n))$ such that 
 $$ \| \nabla \upper{\m}{k}(t)\|_{L^\infty} \le  \frac{c}{\sqrt{t}}   \| \nabla \upper{\m}{k}_0\|_{L^n}   \quad  \text{for all}  \quad t >0$$
and a universal constant $c$.
Passing to the limit $k \to \infty$ we see that by Lemma \ref{lemma:stability} 
$$\upper{\m}{k} \to \m \quad \text{in} \quad  L^\infty_{\rm loc}( (0,T_\ast);L^2(\R^n))$$ 
and by lower semicontinuity of norms 
$$\sup_{t \in (0,T_\ast)} \norm{\nabla \m (t)}{L^\infty} < \infty.$$
We infer that $\m$ persists for all times by virtue of Proposition \ref{prop:local_classic} and is unique in its class according to
Lemma \ref{lemma:stability}. Similarly, \eqref{eq:L_estimate} and \eqref{eq:decay} carry over to the limit.

\subsection{Weak initial data}
Given intitial data $\m_0 : \R^n \to \St$ such that
$$\m_0 - \m_\infty \in H^1 \cap W^{1,n}(\R^n;\R^3)
\quad \text{and} \quad \norm{\nabla \m_0}{L^n} < \rho$$ 
with $\rho$ as in Proposition \ref{prop:alpha}, there exists, by virtue of Lemma \ref{lemma:S_U}, a sequence 
of approximating initial maps $\m^{(k)}_0 \in H^\infty_\ast(\R^n)$ and, for large $k$, corresponding
global solutions $\m^{(k)} \in C^0([0,\infty);H^\infty_\ast(\R^n))$. 
If $c \rho < \eta$ with $\eta$ as in Lemma \ref{lemma:stability}  we deduce in conjunction with the energy inequality \eqref{eq:energy_inequality}
$$  \|  \upper{\m}{k}(t) - \m_\infty \|_{H^1} \le  \|  \upper{\m}{k}_0 - \m_\infty\|_{H^1}$$ 
and
$$ \int_0^\infty \left\| \dd{\upper{\m}{k}}{t} (t)\right\|^2_{L^2} \, dt \le \frac{1+\lambda^2}{\lambda}E(\upper{\m}{k}_0)$$
while from Proposition \ref{prop:alpha}
$$\sqrt{t} \, \| \nabla \m^{(k)}(t) \|_{L^\infty} + \| \nabla \m^{(k)}(t) \|_{L^n} \le c \; \| \nabla \m^{(k)}_0 \|_{L^n}$$
 and from \eqref{eq:decay}
$$  \| \upper{\m}{k}(t) - \m_\infty \|_{L^\infty} \le c \,  t^{-\frac{n-2}{2n}} \, \|\nabla \upper{\m}{k}_0\|_{L^n \cap L^2} $$
for all $t>0$ and with a univeral constant $c$.
Now it is easy to see that any weak limit $\m$ 
is a global weak solution of \eqref{eq:LLG2} such that the prescribed initial data $\m_0$ is continuusly attained in $L^2$ and 
such that all the above estimates are satisfied by $\m$. Hence, by Lemma \ref{lemma:stability}, the solutions $\m$ is
unique in its class.\\
We aim to prove uniform local bounds in spaces of  space-time H\"older continuous functions.
For this purpose we invoke the following localized energy inequality for Landau-Lifshitz-Gilbert, 
that can be found in \cite{Melcher:05} Lemma 2 or \cite{Moser:book} Lemma 5.10:
  \begin{lemma} Suppose $\m$ is a smooth solution of the Landau-Lifshitz-Gilbert equation 
 in a space-time cylinder
 $$ P_r(z_0) =  (t_0, t_0 + r^2) \times B_r(x_0) \quad \text{where} \quad z_0 = (t_0, x_0)$$
 then, for a constant $c$ that only depends on $\lambda$ and $n$,
 \begin{equation}\label{eq:energyinequality}
 \int_{B_{r/2}(x_0)} \left|\nabla \m\right|^2 \, dx
    +   \int_{P_{r/2}(z_0)} \left|\dd{\m}{t}\right|^2 \, dz
	  \le \frac{c}{r^2}  \; \int_{P_r(z_0)}\left|\nabla \m\right|^2 \, dz .
	  \end{equation}
\end{lemma}
Given any compact subset $Q  \subset (0,\infty) \times \R^n$ there exists a constant $c$ that only depends on $\| \nabla \m^{(k)} \|_{L^\infty(Q)}$ such that, for all $P_{r}(z_0) \subset Q$
 $$ \int_{P_{r/2}(z_0)}\left(  |\nabla \m^{(k)}|^2 + r^2 \, \left|\dd{\m^{(k)}}{t}\right|^2 \right) \, dz
   \le c \, r^{2+n}$$
 and hence, by the parabolic version of Morrey's lemma, see \cite{Chen_Lin:93}, a uniform H\"older bound
 of $\{\upper{\m}{k}\}$ locally in $Q$. Hence $\m$ is locally H\"older continuous with locally bounded gradient. But then the linear theory of parabolic systems as e.g. in \cite{Lady}
 and a bootstrap argument implies $\m$ is smooth in $(0,\infty) \times \R^n$. \\
Finally, we show that initial data is strongly attained in $H^1 \cap W^{1,n}(\R^n)$. Indeed, by lower semicontinuity of norms
 and the strong convergence of initial maps
 \begin{eqnarray*}
 \norm{ \m_0-\m_\infty}{}  &\le&  \limsup_{t \searrow 0} \norm{ \m(t)-\m_\infty} {} 
\le \limsup_{t \searrow 0} \liminf_{k \to \infty} \| \upper{\m}{k}(t)-\m_\infty\| \\
&\le& \limsup_{k \to \infty} \|\upper{\m}{k}_0-\m_\infty \| = \| \m_0-\m_\infty \| 
\end{eqnarray*}
where $\|\cdot\|$ is the $H^1 \cap W^{1,n}$ norm.
This completes the proof of the Theorem.
\bibliographystyle{acm}
\bibliography{cm}
   
\end{document}